\documentclass[12pt]{article}
\usepackage[margin=1in]{geometry}
\usepackage{amsmath,amsthm,amssymb}
\usepackage{bbm}
\usepackage{color,colortbl}
\usepackage{subcaption}
\usepackage{multirow,arydshln}
\allowdisplaybreaks

\usepackage{setspace}

\usepackage{tikz}
\usetikzlibrary{arrows,automata,positioning}
\usepackage{pgfplots}
\pgfplotsset{compat=newest}

\newtheorem{lemma}{Lemma}[section]

\newtheorem{theorem}[lemma]{Theorem}
\newtheorem{corollary}[lemma]{Corollary}
\newtheorem{definition}[lemma]{Definition}
\newtheorem{example1}[lemma]{Example}
\newtheorem{rem1}[lemma]{Remark}

\newtheorem{alg1}[lemma]{Algorithm}
\newtheorem{me1}[lemma]{Mechanism}
\newenvironment{remark}{\begin{rem1}\rm}{\end{rem1}}
\newenvironment{example}{\begin{example1}\rm}{\end{example1}}

\newcommand{\bbr}{\mathbb{R}}

\newcommand{\bbx}{\mathbb{X}}

\newcommand{\xcal}{\mathcal{X}}
\newcommand{\ycal}{\mathcal{Y}}
\newcommand{\zcal}{\mathcal{Z}}
\newcommand{\T}{\top}

\newcommand{\vmin}{\min}

\newcommand{\NE}{\operatorname{NE}}

\newcommand{\cl}{\operatorname{cl}}
\newcommand{\co}{\operatorname{co}}
\renewcommand{\int}{\operatorname{int}}

\DeclareMathOperator*{\argmax}{arg\,max}
\DeclareMathOperator*{\argmin}{arg\,min}

\begin{document}

\title{Characterizing and Computing the Set of Nash Equilibria via Vector Optimization}
\author{Zachary Feinstein \thanks{Stevens Institute of Technology, School of Business, Hoboken, NJ 07030, USA, zfeinste@stevens.edu.} \and Birgit Rudloff \thanks{Vienna University of Economics and Business, Institute for Statistics and Mathematics, Vienna A-1020, AUT, brudloff@wu.ac.at.}}
\date{\today~(Original: September 30, 2021)}
\maketitle
\abstract{
Nash equilibria and Pareto optimality are two distinct concepts when dealing with multiple criteria. It is well known that the two concepts do not coincide. However, in this work we show that it is possible to characterize the set of all Nash equilibria for any non-cooperative game as the Pareto optimal solutions of a certain vector optimization problem. To accomplish this task, we increase the dimensionality of the objective function and formulate a non-convex ordering cone under which Nash equilibria are Pareto efficient. We demonstrate these results, first, for shared constraint games in which a joint constraint is applied to all players in a non-cooperative game. In doing so, we directly relate our proposed Pareto optimal solutions to the best response functions of each player. These results are then extended to generalized Nash games, where, in addition to providing an extension of the above characterization, we deduce two vector optimization problems providing necessary and sufficient conditions, respectively, for generalized Nash equilibria. Finally, we show that all prior results hold for vector-valued games as well. Multiple numerical examples are given and demonstrate that our proposed vector optimization formulation readily finds the set of all Nash equilibria. 
}

\section{Introduction}\label{sec:intro}

The Nash equilibrium is a fundamental concept in game theory.  Proposed by John Nash in his seminal papers~(\cite{nash1950,nash1951}), the Nash equilibrium notion provides a stable strategy set in a non-cooperative game setting.  Specifically, these solutions are such that no player can unilaterally reduce her costs when all other players fix their own strategies.  Though an equilibrium of a non-cooperative game was first introduced as the solution to a special linear program by John von Neumann~(\cite{neumann1928theorie}) in a two-player zero-sum setting, a general $N$ player game is solved via fixed point arguments. Even for games with convex costs and constraints, fixed point arguments are required as demonstrated in~\cite{rosen65}; there always exists a Nash equilibrium for a convex game, but stronger convexity conditions are required to guarantee uniqueness.  
In this work, we are not concerned with the existence or uniqueness of Nash equilibria. Instead, our aim is to fully characterize, and thus enable the computation of, the set of all Nash equilibria.  As such, these results are independent of the set of equilibria being a singleton, containing several or even infinitely many elements, or being an empty set.

A Nash game is defined by a collection of optimization problems.  Mathematically this is encoded by cost functions and constraint sets for each player. Though it is tempting to consider the vector optimization problem constructed from the $N$ player non-cooperative game in which all players simultaneously attempt to minimize their costs, 
the set of minimizers of this vector optimization problem generally does \emph{not} coincide with the set of the Nash equilibria for the game.
Our primary motivation within this work is to construct a particular vector optimization problem that provides exactly the set of Nash equilibria as its minimizers.  This allows us to think about Nash equilibria as just a special type of Pareto efficiency. 

Importantly, where the appropriate vector optimization algorithms exist, this new optimization construction of the set of Nash equilibria allows for a novel computational technique that provides \emph{all} Nash equilibria. This proposed technique does not require any fixed point iterations.
For example, algorithms -- like those proposed in \cite{Armand93,VanTu17,tohidi2018adjacency} -- can be used to solve the corresponding linear vector optimization problems appearing in the Pareto characterization of the set of Nash equilibria for games with linear objectives and constraints.
For convex and non-convex games, the corresponding vector optimization problems can in practice only be solved approximately. The corresponding relation between epsilon Pareto optimal solutions and epsilon Nash equilibria and the development of numerical algorithms is ongoing work, for the convex case see~\cite{HFR22}. The present paper will therefore focus on the equivalent characterization holding for all non-coorporative games and illustrates the theoretical results both on games that can be solved analytically and on linear games, where existing vector optimization algorithm can be applied to compute the set of Nash equilibria.

The problem of finding a single Nash equilibrium has been well-studied for (nearly) all classes of games; for instance,~\cite{rosen65} provides an efficient algorithm to compute the unique equilibrium for a class of convex games.  However, for games with multiple equilibria, understanding the entire set of Nash equilibria is vitally important to fully characterize the possible (jointly) rational behaviors of the players in the game.  Despite the importance of finding all Nash equilibria in practice, this problem has received less attention and results are so far limited to specific classes of games.  There have been numerous studies on computing all Nash equilibria for finite games; we refer the interested reader to e.g.,~\cite{mckelvey1996computation,parsopoulos2004computation,herings2005globally,wu2015new}, references therein, and citations thereof.  We also wish to highlight~\cite{judd2012finding} which utilizes algebraic geometry to find (a superset of) all Nash equilibria in continuous games without binding constraints, i.e., where the equilibria lie within the interior of the feasible region.  Generalized Nash games have additionally been approached as quasi-variational inequalities~(\cite{harker1991}); such an approach allows for the computation of (a superset of) all Nash equilibria by considering an infinite number of variational inequalities in a shared constraint setting~(\cite{nabetani2011parametrized}).  In contrast to these aforementioned approaches, our methodology merely relies on the appropriate vector optimization algorithms to compute all Nash equilibria.  In fact, our approach can be readily extended to vector games (\cite{Corley85,bade2005}) in which the computation of Nash equilibria is often considered a challenge because of the large number of solutions and for which, as far as the authors are aware, there do no exist any algorithms to compute all equilibria for general classes of problems.  To highlight the novelty of our approach, we provide numerical examples in which we compute all Nash equilibria in settings where these prior cited approaches \emph{cannot} be applied.

The primary contributions of this work are three-fold. First, we provide 
an equivalent characterization of the set of all Nash equilibria of a game as Pareto optimal points of a certain vector optimization problem. This holds for all non-cooperative games without any further assumptions. 
We highlight in Remark~\ref{Rem:BestResp} a relation between the solutions of the specified vector optimization problem and the best response functions of all players; this connects our results to the well-known simple characterization of the set of Nash equilibria as the intersection of the graphs of the best response functions for all players (see e.g.\ the introduction section in~\cite{BeckStein22}).
We will illustrate this characterization on a simple game with non-convex cost functions.
Second, for generalized Nash games, we construct two bounding vector optimization problems that provide, respectively, necessary and sufficient conditions for the Nash equilibria  and correspond to specific shared constraint games considering the union or intersection of the players' constraints, respectively.
Furthermore, and as a third contribution, this vector optimization reformulation of the Nash game permits us to consider optimization-based methods for finding \emph{all} Nash equilibria. 
This is in contrast to the typical approaches that construct just a single Nash equilibrium as well as the aforementioned works on the computation of all Nash equilibria for specific classes of games.
This extends well beyond finite games to provide the theoretical foundation for a numerical technique to find all Nash equilibria for, e.g., convex games or  even for vector-valued games.

The organization of this paper is as follows.  
In Section~\ref{sec:shared}, the set of Nash equilibria for a Nash game with a shared constraint are completely characterized as the Pareto solutions to an appropriate vector optimization problem with a \emph{non-convex} ordering cone.  Due to the special structure of this cone, we present a decomposition that allows for its tractable use.  We extend this result in Section~\ref{sec:generalized} to study the use of the proposed vector optimization problems for generalized Nash games. In such a setting, we first construct a characterization of the set of Nash equilibria.  Second, we use our vector optimization framework to construct bounds around the set of Nash equilibria. Specifically, we find that one proposed vector optimization problem leads to a necessary condition, a different one to a sufficient condition for the Nash equilibria. 
We further extend these results in Section~\ref{sec:vector} to study vector-valued games; importantly, all of the prior results hold under only minimal modifications to the ordering cone.  
We present a number of illustrative examples throughout this work to demonstrate the application of the proposed results for the computation of the set of Nash equilibria.

\section{Shared Constraint Nash Games as Vector Optimization}\label{sec:shared}

In this section, we will construct a particular vector optimization problem whose Pareto optimal points coincide with, and thus equivalently characterize, the set of equilibria of shared constraint Nash games.
Section~\ref{sec:background} will introduce the basic notation and definitions of Nash equilibria, vector optimization and Pareto optimal points. Section~\ref{sec:NashPareto} will contain the main results of this section on the characterization of the set of equilibria of shared constraint Nash games as Pareto optimal points of a certain vector optimization problem.
In Section~\ref{sec:NashEx}, we will illustrate these results with examples.

\subsection{Background and Notation}\label{sec:background}
Let us now introduce the two disparate frameworks for studying competing interests in optimization -- Pareto optimality and Nash equilibria.  In doing so, we will provide basic definitions and notation which will be utilized for the remainder of this work.

Throughout this work we will study $N \geq 2$ player non-cooperative games.  We will occasionally refer to these as Nash games as the relevant solution concept is that of a Nash equilibrium (see Definition~\ref{defn:nash-shared} below).  Player $i$ (for $i = 1,...,N$) of this game considers strategies in the linear space $\xcal_i$; we do not impose any condition that $\xcal_i$ and $\xcal_j$ are equal.
In addition, each player $i$ has a cost function $f_i: \prod_{j = 1}^N \xcal_j \to \bbr$ she seeks to minimize.  This cost function $f_i(x)$ for $x \in \prod_{j = 1}^N \xcal_j$ may depend on player $i$'s strategy $x_i \in \xcal_i$ as well as the strategy chosen by all other players $x_{-i} \in \prod_{j \neq i} \xcal_j$.
Finally, in order to complete the setup of the $N$ player non-cooperative game, we wish to introduce the notion of a shared constraint $\bbx \subseteq \prod_{j = 1}^N \xcal_j$ which must be satisfied by the joint strategy $x \in \prod_{j = 1}^N \xcal_j$ of all players.  This shared constraint condition is common in the literature (see, e.g., \cite{rosen65,facchinei2007generalized}). 
\begin{remark}\label{rem:pure-mixed}
By utilizing the constraint set $\bbx \subseteq \prod_{j = 1}^N \xcal_j$ in the product of the linear spaces $\xcal_i$ we allow for either pure or mixed strategy Nash equilibria.  Furthermore, the use of the general constraint set $\bbx$ permits us to study, e.g., both discrete or continuous games.
For example, in a two player game in which both players can only select in the strategy set $\{0,1\}$, then $\xcal_1 = \xcal_2 =\bbr$ with $\bbx = \{0,1\}^2$. If one allows mixed strategies in this game, which are encoded by the probability $x_i \in [0,1]$ as the probability of selecting strategy $1$, one would set $\mathbb{X}= [0,1]^2$.
\end{remark}

Given all other players fix their strategies $x_{-i}^* \in \prod_{j \neq i} \xcal_j$, player $i$ seeks to optimize
\begin{equation}\label{eq:game-shared}
\min_{x_i \in \xcal_i}\left\{f_i(x_i,x_{-i}^*) \; | \; (x_i,x_{-i}^*) \in \bbx\right\}.
\end{equation}
If this is the aim of all players $i = 1,...,N$, this leads to the definition of a Nash equilibrium.
\begin{definition}\label{defn:nash-shared}
Consider the shared constraint Nash game with $N$ players described by~\eqref{eq:game-shared}.  
That is, player $i$ has strategies in the space $\xcal_i$, cost function $f_i: \prod_{j = 1}^N \xcal_j \to \bbr$, and such that the joint strategies must satisfy a shared feasibility condition $x \in \bbx \subseteq \prod_{j = 1}^N \xcal_j$.  
A joint strategy $x^* \in \bbx$ is called a \textbf{\emph{Nash equilibrium}} if, for any player $i$, 
$f_i(x_i,x_{-i}^*) \geq f_i(x^*)$ for any strategy $x_i \in \xcal_i$ such that $(x_i,x_{-i}^*) \in \bbx$.  The set of all Nash equilibria will be denoted by $\NE(f,\bbx)$.
\end{definition}

Later in this work we will relax the assumption of the shared constraint to study generalized Nash games (see Section~\ref{sec:generalized}).  We will, additionally, relax the assumption that the cost functions are real-valued to study vector-valued games in Section~\ref{sec:vector}.

The main focus of this work is on the relation between the set of Nash equilibria and the Pareto optimal points of a certain vector optimization problem.
Let us now introduce the basic notations needed for that. Consider a linear space $\zcal$. A set $C \subseteq  \zcal$ is called a cone if $\alpha C\subseteq C$ for all $\alpha \geq 0$. A cone $C \subseteq  \zcal$ introduces 
a binary relation 
$\leq_C$ on $ \zcal$ via
\[
	x\leq_C y \iff y-x\in C,
\]
for $x,y\in  \zcal$. This relation is reflexive. If the cone $C$ is additionally convex, i.e., if it holds $C+C\subseteq C$, then the relation $\leq_C$ is reflexive and transitive and thus defines a pre-order relation on  $ \zcal$. If $C$ is not convex, then the relation $\leq_C$ is not transitive and so it
is not a pre-order relation and is sometimes called a `non-convex preference' (see~\cite{anderson1985strong} for a discussion of non-convex preferences). We will also in that case refer to the cone $C$ as the ordering cone of $\zcal$. 
From now on we consider a linear space $\zcal$ endowed with an ordering cone $C$.
When $\zcal=\bbr$ and $C=\bbr_+$, we will leave off the explicit notation for the cone $C$, i.e., $x \leq y$ if and only if $y-x \in \bbr_+$ for $x,y \in \bbr$.

A vector optimization problem is an optimization problem of the form 
\begin{equation}\label{eq:VOP}
	C\mbox{-}\vmin \{g(x) \; | \; x \in \bbx\} 
\end{equation}
 for some feasible region $\bbx \subseteq \xcal$, a function $g: \bbx \to  \zcal$ and an ordering cone $C \subseteq  \zcal$. 
 Note that we do not necessarily assume the cone $C$ to be convex.  
The notation $C\mbox{-}\vmin$ used in~\eqref{eq:VOP} means that one is minimizing the vector-valued function $g$, where the usage of $C$ in this notation emphasizes that one minimizes with respect to the binary relation introduced by the ordering cone $C$. To minimize this vector-valued function, and thus to solve the vector optimization problem~\eqref{eq:VOP}, means to compute the set of minimizers, also called Pareto optimal points, or efficient points, which are defined as follows.  We use the notation $g[\bbx]:=\{g(x) \; | \; x\in \bbx\}  \subseteq  \zcal$ for the image of the feasible set $\bbx$.
  \begin{definition}\label{defn:pareto}
The set of minimizers or Pareto optimal points of problem~\eqref{eq:VOP} 
is defined as 
\begin{align*}
	C\mbox{-}\argmin \{g(x) \; | \; x \in \bbx\} 
    &= \{x^* \in \bbx \; | \; g(x^*) \leq_C g(x) \; \forall x \in \bbx \; \text{with} \; g(x) \leq_C g(x^*)\}.
\end{align*}
This is the set of feasible points $x^*$ that map to minimal elements of $g[\bbx]$, i.e., the set of points $x^* \in \bbx$ such that if $g(x) \leq_C g(x^*)$ for some $x \in \bbx$ then $g(x^*) \leq_C g(x)$ holds.
\end{definition}
The above definitions can be found in~\cite[Def.\ 3.1(c), Def.\ 4.1(a) and Def.\ 7.1(a)]{Jahn11} for the case of a convex cone $C$ and partially ordered normed space $\zcal$ (but can also be stated on the general level provided here, i.e., for a cone $C$ and linear space $\zcal$). They also coincide with the more general definitions given in \cite[Chapter 6]{TW20} by setting the reference set there to be $D =C\backslash(-C)$.

\subsection{Nash equilibria as Pareto optimal points}\label{sec:NashPareto}
In this section, we want to deduce a relation between the set of Nash equilibria and the Pareto optimal points of a certain vector optimization problem.
Note that the Definition~\ref{defn:nash-shared} of a Nash equilibrium can equivalently be written as: the joint strategy $x^* \in \bbx$ is called a Nash equilibrium if, for any player $i$, 
\begin{equation}\label{eq:Nashi}
	x_i^* \in \argmin_{x_i \in \xcal_i}\{f_i(x_i,x_{-i}^*) \; | \; (x_i,x_{-i}^*) \in \bbx\}.
\end{equation}
Naively, one may attempt to construct the vector optimization problem
\begin{equation}\label{eq:wrongPareto}
	\bbr^N_+\mbox{-}\vmin \{f(x) \; | \; x \in \bbx\}
\end{equation}
with $f(x) := (f_1(x),f_2(x),\cdots,f_N(x))$ from the $N$ player non-cooperative game. However, it is well known that the minimizers of this vector optimization problem generally do \emph{not} coincide with any of the Nash equilibria for the game (see, e.g.,~\cite[Chapter 3.3]{tadelis2013game}). This is because, in the underlying optimization problem for each player $i$ in~\eqref{eq:Nashi}, the strategy $x_{-i}^*$ of all other players is fixed.
We will now construct a vector optimization problem that will exactly take these considerations into account.

In order to facilitate this, we will need two modifications to~\eqref{eq:wrongPareto}. First, we need to put the strategy in the objective as well such that one can consider, for each player, $x_{-i}$ and $f_i(x)$ at the same time. That is, we will set $g(x)=(x,f(x))$ as the objective function in~\eqref{eq:VOP} instead of $g(x)=f(x)$, as was done in~\eqref{eq:wrongPareto}. Thus, the linear space $\zcal$ in~\eqref{eq:VOP} is chosen to be $\prod_{i = 1}^N \xcal_i \times \bbr^N$.  Second, we will define a particular ordering cone that can `fix' the component $x_{-i}$ of the vector $x$ for player $i$. That is, we will consider the following vector optimization problem 
\begin{align}
\label{eq:vop} C\mbox{-}\vmin  \{(x,f(x)) \; | \; x \in \bbx\} 
\end{align}
with ordering cone
\begin{align}
\label{eq:C} C &:= \{(x,y) \in \prod_{i = 1}^N \xcal_i \times \bbr^N \; | \; \exists i \in \{1,...,N\}: \; x_{-i} = 0, \; y_i \geq 0\},
\end{align}
where the $0$ in $x_{-i} = 0$ is the zero in the space $\prod_{j\neq i}\xcal_j$. Though the ordering cone $C$ is \emph{not} convex, it is of the form $C = \bigcup_{i = 1}^N C_i$ for convex cones $C_i$ defined as
\begin{equation}\label{eq:Ci}
C_i := \left(\prod_{j = 1}^{i-1} \{0_{\xcal_j}\}\right) \times \xcal_i \times \left(\prod_{j = i+1}^N \{0_{\xcal_j}\}\right) \times \bbr^{i-1} \times \bbr_+ \times \bbr^{N-i},
\end{equation}
where $0_{\xcal_j}$ denotes the zero in the space $\xcal_j$. Whenever the space is clear from the context, we will, for the remainder of this work, drop the subscript notation and just use $0$ for the zero in the corresponding space.
The decomposition of $C$ into the cones $C_i$ will be utilized in the following and plays a particular role in the examples via Corollary~\ref{cor:nash-shared}.
The following remarks motivate and illustrate the choice of the vector optimization problem~\eqref{eq:vop} and its cone~\eqref{eq:C}.
\begin{example}
Let us, for illustration purposes, consider the cone $C$ in the easiest setup of a 2-player game ($N=2$) with $\xcal_1=\xcal_2=\bbr$. Then, $C\subseteq \bbr^4$ is given by
\begin{align}
	\nonumber C&= \{(x,y) \in  \bbr^2\times \bbr^2 \; | \; \exists i \in \{1,2\}: \; x_{-i} = 0, \; y_i \geq 0\}
	\\
	\label{exC_1}&=(\bbr\times \{0\}\times \bbr_+\times \bbr)\cup(\{0\}\times \bbr\times \bbr\times \bbr_+)=C_1\cup C_2.
\end{align}
For a 2-player game ($N=2$) with $\xcal_1=\xcal_2=\bbr^2$, the cone $C\subseteq \bbr^6$ is given by
\begin{align}
	C
	\label{exC_2}&=(\bbr^2\times \{(0,0)\}\times \bbr_+\times \bbr)\cup(\{(0,0)\}\times \bbr^2\times \bbr\times \bbr_+)=C_1\cup C_2.
\end{align}	
\end{example}

\begin{remark}
Note the roles the different components in the cones $C$, respectively $C_i$ play. Consider the following three ordering cones in $\bbr$: 
\begin{itemize}
\item $\bbr_+$ is the natural ordering cone and characterizes the usual order relation $x\leq y\iff x\leq_{\bbr_+} y\iff y-x\in \bbr_+$ that is, e.g., considered in player $i$'s optimization problem~\eqref{eq:game-shared}, see also \eqref{eq:Nashi}, when player $i$ minimizes her objective function $f_i$.
\item $ \{0\}$ is an ordering cone that allows us to `fix' the corresponding component in the sense that $x\leq_{\{0\}}  y \iff y-x \in \{0\}\iff x=y$. That is, the only element that can be compared to $x$ is $x$ itself. This becomes important since we need, in player $i$'s optimization problem~\eqref{eq:game-shared}, to fix the other players strategies to be $x_{-i}^*$. 
\item $\bbr$ is an ordering cone that allows us to ignore the corresponding component in the sense that $x\leq_{\bbr}  y \iff y-x \in\bbr\iff x\geq_{\bbr}  y $ is true for all $x,y\in\bbr$. That is, with respect to this ordering cone, every point would be minimal. Thus, with respect to the corresponding component, where this ordering cone is considered, nothing needs be optimized, so this component is ignored. This becomes important as player $i$ does not care about the cost functions of his opponents in his optimization problem~\eqref{eq:game-shared}. Additionally, the full space also appears as the ordering cone in the component concerned with the strategy $x_i$ of player $i$. In such a capacity, this ordering cone models the idea that there are no further restrictions player $i$ imposes on his own strategy (besides the constraint set $\bbx$ that is incorporated already in problem~\eqref{eq:vop}).
\end{itemize}
The interpretations of the last two cones above hold analogously if one considers the cones $ \{(0,0)\}$ or $\bbr^2$ if $\xcal_i=\bbr^2$, or for more general $\xcal_i$ the cones  $\{0_{\xcal_i}\}$ or $\xcal_i$.

Consider now the specific structure of the preferences induced by the ordering cone $C$ and its decomposition $C = \bigcup_{i = 1}^N C_i$.  These structures are central to relating the set of Nash equilibria to the Pareto optimal points of~\eqref{eq:vop} as is provided in Theorem~\ref{thm:nash-shared} below.  Let us first study the cone $C_i$: consider $x,\bar x \in \prod_{i = 1}^N \xcal_i$, then $(x,f(x)) \geq_{C_i} (\bar x,f(\bar x))$ for some player $i$ if and only if $x_{-i} = \bar x_{-i}$ and $f_i(x) \geq f_i(\bar x)$.  That is, for a `fixed' strategy of all other players, the cost of player $i$ is greater under $x$ than $\bar x$.  For the full cone $C$, the relation $(x,f(x)) \geq_C (\bar x,f(\bar x))$ holds if and only if there exists some player $i$ such that $(x,f(x)) \geq_{C_i} (\bar x,f(\bar x))$.  From this, the correspondence to the notion of Nash equilibrium becomes clear as both the Nash and the Pareto notion with ordering cone $C$ focus on a comparison in which the strategies of all other players have been `fixed'.
\end{remark}

We can now present the main result of this work -- the equivalence between the set of all Nash equilibria for the $N$ player game~\eqref{eq:game-shared} and the Pareto minimizers of the vector optimization problem~\eqref{eq:vop}.
\begin{theorem}\label{thm:nash-shared}
Consider the shared constraint non-cooperative game \eqref{eq:game-shared}.  
The strategy $x^*$ is a Nash equilibrium if and only if it is a minimizer of~\eqref{eq:vop}, i.e., 
\[
	\NE(f,\bbx) = C\mbox{-}\argmin \{(x,f(x)) \; | \; x \in \bbx\},
\] 
where the ordering cone $C$ is defined in~\eqref{eq:C}.
\end{theorem}

To prove the theorem, we will need the following lemma. Recall from Definition~\ref{defn:pareto} that $x^* \in C\mbox{-}\argmin \{(x,f(x)) \; | \; x \in \bbx\}$ if and only if for any $x \in \bbx$ such that $(x,f(x)) \leq_C (x^*,f(x^*))$ then $(x,f(x)) \geq_C (x^*,f(x^*))$.
\begin{lemma}\label{lemma:decomposition} It holds 
    $C\mbox{-}\argmin \{(x,f(x)) \; | \; x \in \bbx\}= \bigcap_{i = 1}^N C_i\mbox{-}\argmin \{(x,f(x)) \; | \; x \in \bbx\},$
where $C$ and $C_i$ are defined in~\eqref{eq:C} and~\eqref{eq:Ci}.
\end{lemma}
\proof{Proof}
Let $x^*\in C\mbox{-}\argmin \{(x,f(x)) \; | \; x \in \bbx\}$ and assume $x^* \notin C_i\mbox{-}\argmin \{(x,f(x)) \; | \; x \in \bbx\}$ for some player $i$.
Therefore, by the construction of the $C_i$ minimizers, there exists some $x \in \bbx$ such that $x_{-i} = x_{-i}^*$ with $f_i(x) < f_i(x^*)$; by construction of the cone $C$ this implies $(x,f(x)) \leq_C (x^*,f(x^*))$ as well.
However, as $x^*$ is a $C$ minimizer, it must follow that $(x,f(x)) \geq_C (x^*,f(x^*))$, i.e., there exists some player $j$ such that $x_{-j} = x_{-j}^*$ and $f_j(x) \geq f_j(x^*)$.  If $j \neq i$ then $x = x^*$ (i.e., $f_i(x) = f_i(x^*)$); if $j = i$ then $f_i(x) \geq f_i(x^*)$.  In either case we recover a contradiction.

Now consider  $x^*\in \bigcap_{i = 1}^N C_i\mbox{-}\argmin \{(x,f(x)) \; | \; x \in \bbx\}$ and assume that $x^*\notin C\mbox{-}\argmin \{(x,f(x)) \; | \; x \in \bbx\}$. That is, there is an $x \in \bbx$ with $(x,f(x)) \leq_{C_i} (x^*,f(x^*))$ for some $i$ 
such that there is no $j=1,...,N$ with $(x,f(x)) \geq_{C_j} (x^*,f(x^*))$. In particular $(x,f(x)) \geq_{C_i} (x^*,f(x^*))$ does not hold, which contradicts $x^*\in C_i\mbox{-}\argmin \{(x,f(x)) \; | \; x \in \bbx\}$.
\endproof

\proof{Proof of Theorem~\ref{thm:nash-shared}}
Consider $x^*$ to be a Nash equilibrium of the game.  By definition of the Nash equilibrium, for any player $i$, if $x \in \bbx$ with $x_{-i} = x_{-i}^*$ then $f_i(x) \geq f_i(x^*)$.  Therefore if $(x,f(x)) \leq_C (x^*,f(x^*))$ then $x_{-i} = x_{-i}^*$ for some $i$ and, from being a Nash equilibrium, we know $f_i(x) \geq f_i(x^*)$, i.e., $(x,f(x)) \geq_C (x^*,f(x^*))$.

Now consider $x^*$ be a minimizer of the vector optimization problem, i.e.,  $x^*\in C\mbox{-}\argmin \{(x,f(x)) \; | \; x \in \bbx\}$.  For each player $i$ consider $x_{-i} = x_{-i}^*$ with $x \in \bbx$ then if $f_i(x) \leq f_i(x^*)$, the minimality of $x^*$ and Lemma~\ref{lemma:decomposition} imply 
 $f_i(x) \geq f_i(x^*)$. 
 This is the definition of the Nash equilibrium, i.e., for every player $i$: $(x_i,x_{-i}^*) \in \bbx$ implies $f_i(x_i,x_{-i}^*) \geq f_i(x^*)$.
\endproof

The following corollary is immediate and shows that $x^*$ is a Nash equilibrium if and only if it is a minimizer with respect to $C_i$ for each player $i$.  We wish to highlight this result, as it can be directly expanded to generalized games, see Section~\ref{sec:genNash}.  Additionally, this construction is useful in the subsequent examples, as each of these cones $C_i$ is convex in contrast to the full ordering cone $C$ that is non-convex.  This allows for the application of traditional vector optimization algorithms as used, e.g., in Examples~\ref{ex:generalized} and~\ref{ex:vector} in the more general settings.
\begin{corollary}\label{cor:nash-shared}
$x^*$ is a Nash equilibrium if and only if $x^* \in C_i\mbox{-}\argmin \{(x,f(x)) \; | \; x \in \bbx\}$ for every $i = 1,...,N$, i.e., 
\[
	\NE(f,\bbx) = \bigcap_{i = 1}^N C_i\mbox{-}\argmin \{(x,f(x)) \; | \; x \in \bbx\}.
\]
\end{corollary}
\proof{Proof}
This follows immediately by Theorem~\ref{thm:nash-shared} and Lemma~\ref{lemma:decomposition}.
\endproof

\begin{remark}\label{Rem:BestResp}
As will be made explicit in Examples~\ref{ex:odd-even} and~\ref{ex:shared-multiple} below, the minimizers $C_i\mbox{-}\argmin\{(x,f(x)) \; | \; x \in \bbx\}$ with respect to player $i$'s ordering cone $C_i$ exactly coincide with the graph of player $i$'s best response function $B_i: \prod_{j \neq i} \xcal_j \to 2^{\xcal_i}$: \[\operatorname{graph} B_i := \left\{x^* \in \bbx \; | \; x_i^* \in \argmin_{x_i \in \xcal_i} \{f_i(x_i,x_{-i}^*) \; | \; (x_i,x_{-i}^*) \in \bbx\}\right\}.\]
As such, the result presented in Corollary~\ref{cor:nash-shared} can be reformulated as the well known characterization (see e.g.~\cite{BeckStein22})
\[\NE(f,\bbx) = \bigcap_{i = 1}^N \operatorname{graph} B_i,\]
but it now also allows one to compute $\operatorname{graph} B_i$ via vector optimization techniques.
In this way, Corollary~\ref{cor:nash-shared} can also be seen through the lens of finding a fixed point of the best response functions $x^* \in B(x^*) := (B_1(x_{-1}^*) , ... , B_N(x_{-N}^*))$.

Furthermore for a differentiable and convex game (i.e., $f_i$ is differentiable and convex for every player $i$ and $\bbx$ is convex), the first order sufficient condition of the vector optimization problem $C_i\mbox{-}\vmin\{(x,f(x)) \; | \; x \in \bbx\}$ is equivalent to the first order sufficient condition for the best response function for player $i$.  
That is, by modification of Theorem 12.3.1 of \cite{khan2015} (to account for the possibility that $C_i \cap (-C_i) \supseteq \{0\}$) and setting $g(x) := (x,f(x))$,
\begin{align*}
D(g + C_i)(x^*,g(x^*))(x - x^*) \cap -C_i \subseteq C_i \quad \forall x \in \bbx,
\end{align*}
where
$
D(g + C_i)(x^*,g(x^*))(\bar x):= \{y \in \ycal_i \; | \; \exists t_n \searrow 0, \, (x_n,y_n) \to (\bar x,y): \; g(x^*) + t_n y_n \in g(x^* + t_n x_n) + C_i\}
$
is the contingent derivative of the upper image of $g$ at $(x^*,g(x^*))$, i.e., the mapping whose graph is the cone tangent to the graph of the upper image of $g$ at $(x^*,g(x^*))$.
This first order condition is satisfied if and only if $\nabla_{x_i}f_i(x^*)^\T(x_i-x_i^*) \geq 0$ for every $x_i \in \xcal_i$ such that $(x_i,x_{-i}^*) \in \bbx$. This follows from the particular structure of the cone $C_i$ and properties of the contingent derivative.
We note that such results can be compared with, e.g., the (quasi)-variational inequality sufficient condition for Nash games presented in, e.g.,~\cite{harker1991,facchinei2007generalized}.
\end{remark}

\subsection{Examples}\label{sec:NashEx}
We will now illustrate the characterization of the set of Nash equilibria as the intersection of minimizers of $N$ vector optimization problems given in Corollary~\ref{cor:nash-shared} on two examples. We chose two simple examples in which the set of Nash equilibria can easily be deduced and the set of minimizers of the vector optimization problems can be easily depicted. More complex examples, where the characterization of the set of Nash equilibria as the intersection of minimizers of $N$ vector optimization problems is computationally exploited, will be presented in the more general settings of Sections~\ref{sec:generalized} and~\ref{sec:vector}.

\begin{example}\label{ex:odd-even}
Consider the zero-sum $2$ player game of odds and evens.  In this game, both players must choose to say $0$ or $1$ simultaneously; if the sum is even then player $1$ receives a payout of $1$ from player $2$ and vice versa if the sum is odd.  For this game we will consider the mixed strategies so that player $1$ says $1$ with probability $x_1 \in [0,1]$ and similarly for player $2$ with probability $x_2 \in [0,1]$.  The cost functions $f_1: \bbr^2 \to \bbr$ and $f_2: \bbr^2 \to \bbr$, providing the expected losses of each player, are given by
\begin{align*}
f_1(x) &= -x_1 x_2 - (1-x_1)(1-x_2) + x_1 (1-x_2) + (1-x_1) x_2 = -1 + 2x_1 + 2x_2 - 4x_1 x_2\\
f_2(x) &= x_1 x_2 + (1-x_1)(1-x_2) - x_1 (1-x_2) - (1-x_1) x_2 = 1 - 2x_1 - 2x_2 + 4x_1 x_2.
\end{align*}
As the strategies consist of choosing the probabilities $x_i$, the shared constraint set is $\bbx := [0,1]^2$ with real-valued strategies $\xcal_1 = \xcal_2 = \bbr$. The cones $C$ and $C_i$ are given as in~\eqref{exC_1} with $\zcal=\bbr^4$. Note that since the objective functions are not convex, the subproblems $C_i\mbox{-}\argmin \{(x,f(x)) \; | \; x \in \bbx\}$, despite having a convex ordering cone $C_i$ for $i=1,2$, are \emph{not} convex vector optimization problems. However, as they are easy to solve, one obtains that the set of Pareto optimal points of problem $C_1\mbox{-}\argmin \{(x,f(x)) \; | \; x \in \bbx\}\subseteq\bbr^2$, also depicted in Figure~\ref{fig:odd-even-1}, is given by 
\begin{align*}
C_1&\mbox{-}\argmin \{(x,f(x)) \; | \; x \in \bbx\} = \left\{(0,x_2) \; | \; x_2 \in [0,0.5)\right\} \cup \left\{(x_1,0.5) \; | \; x_1 \in [0,1]\right\} \cup \left\{(1,x_2) \; | \; x_2 \in (0.5,1]\right\}.
\end{align*}
Similarly, the set of Pareto optimal points of problem $C_2\mbox{-}\argmin \{(x,f(x)) \; | \; x \in \bbx\}$, which is depicted in Figure~\ref{fig:odd-even-2}, is given by
\begin{align*}
C_2&\mbox{-}\argmin \{(x,f(x)) \; | \; x \in \bbx\} = \left\{(x_1,1) \; | \; x_1 \in [0,0.5)\right\} \cup \left\{(0.5,x_2) \; | \; x_2 \in [0,1]\right\} \cup \left\{(x_1,0) \; | \; x_1 \in (0.5,1]\right\}.
\end{align*}
From this we find that the intersection of these two sets of Pareto optimal points, and thus by Corollary~\ref{cor:nash-shared} the set of Nash equilibria as well as the set of Pareto optimal points of $C\mbox{-}\argmin \{(x,f(x)) \; | \; x \in \bbx\}$, is given by
\begin{align*}
\NE(f,\bbx) =\bigcap_{i = 1}^2 C_i\mbox{-}\argmin \{(x,f(x)) \; | \; x \in \bbx\} &= \{(0.5,0.5)\}.
\end{align*}
Thus, in this example there is a unique Nash equilibrium and this equilibrium consists of the optimal strategy for both players to say $0$ or $1$ randomly with probability $0.5$.
This solution can be verified against the minimax approach pioneered from John von Neumann~(\cite{neumann1928theorie}) as this is a two-player zero-sum finite game.

Recall from Remark~\ref{Rem:BestResp} that the Pareto optimal points $C_i\mbox{-}\argmin\{(x,f(x)) \; | \; x \in \bbx\}$ provide the graph of the best response function for player $i$. This can be seen directly in Figures~\ref{fig:odd-even-1} and~\ref{fig:odd-even-2}.
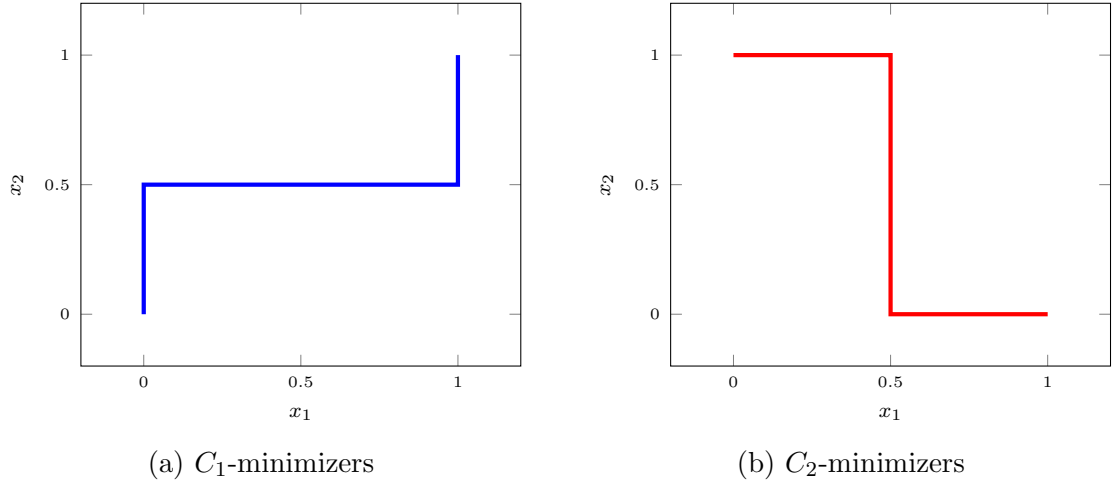
\begin{figure}[t]
\centering
\begin{subfigure}[t]{0.45\textwidth}
\centering
\begin{tikzpicture}
\begin{axis}[
    width=\textwidth,
    xlabel={$x_1$},
    ylabel={$x_2$},
    xmin=-0.2, xmax=1.2,
    ymin=-0.2, ymax=1.2,
    label style={font=\scriptsize},
    ticklabel style={font=\tiny},
    legend pos={north west},
    legend style={font=\scriptsize},
    ]
\addplot [color=blue,solid,ultra thick] coordinates {
    (0,0)
    (0,0.5)
    (1,0.5)
    (1,1)};
\end{axis}
\end{tikzpicture}
\caption{$C_1$-minimizers}
\label{fig:odd-even-1}
\end{subfigure}
~
\begin{subfigure}[t]{0.45\textwidth}
\centering
\begin{tikzpicture}
\begin{axis}[
    width=\textwidth,
    xlabel={$x_1$},
    ylabel={$x_2$},
    xmin=-0.2, xmax=1.2,
    ymin=-0.2, ymax=1.2,
    label style={font=\scriptsize},
    ticklabel style={font=\tiny},
    legend pos={north west},
    legend style={font=\scriptsize},
    ]
\addplot [color=red,solid,ultra thick] coordinates {
    (0,1)
    (0.5,1)
    (0.5,0)
    (1,0)};
\end{axis}
\end{tikzpicture}
\caption{$C_2$-minimizers}
\label{fig:odd-even-2}
\end{subfigure}
\caption{Minimizers for the decomposed vector optimization problem associated with the ordering cone $C_i$ for player $i$ in Example~\ref{ex:odd-even}.}
\label{fig:odd-even}
\end{figure}
\end{example}

\begin{example}\label{ex:shared-multiple}
We now consider a modification to the zero-sum game presented in Example~\ref{ex:odd-even} to make it a general-sum game and such that the set of Nash equilibria will no longer be a singleton.  Specifically, player $1$ is exactly as in Example~\ref{ex:odd-even},  whereas player 2 has the new cost function so that she is penalized for diverging in her strategy from player $1$.  Mathematically this is encoded in the game with cost functions
\begin{align*}
f_1(x) &= 
-1 + 2x_1 + 2x_2 - 4x_1 x_2,
\qquad f_2(x) = |x_1 - x_2|
\end{align*}
and constraint set $\bbx = [0,1]^2$ with real-valued strategies $\xcal_1 = \xcal_2 = \bbr$. $C, C_i$ and $\zcal$ are as in Example~\ref{ex:odd-even} above.  
The set $C_1\mbox{-}\argmin \{(x,f(x)) \; | \; x \in \bbx\}$ of the (\emph{non}-convex) first problem is exactly as in Example~\ref{ex:odd-even}.
The set of minimizers of the ($C_2$-convex) second problem is
\begin{align*}
C_2\mbox{-}\argmin \{(x,f(x)) \; | \; x \in \bbx\} &= \{(x,x) \; | \; x \in [0,1]\}.
\end{align*}
Both are depicted in Figure~\ref{fig:shared-multiple}.
Their intersection is, by Corollary~\ref{cor:nash-shared}, the set of Nash equilibria
\begin{align*}
\NE(f,\bbx) =\bigcap_{i = 1}^2 C_i\mbox{-}\argmin \{(x,f(x)) \; | \; x \in \bbx\} &= \{(0,0) \; , \; (0.5,0.5) \; , \; (1,1)\}.
\end{align*}
By Theorem~\ref{thm:nash-shared} or Lemma~\ref{lemma:decomposition}, this also coincides with the set of Pareto optimal points of the vector optimization problem $C\mbox{-}\argmin \{(x,f(x)) \; | \; x \in \bbx\}$.
Thus, this Nash game has three equilibria, which can also be seen as the three intersection points marked in black in Figure~\ref{fig:shared-multiple}.
\begin{figure}[t]
\centering
\begin{tikzpicture}
\begin{axis}[
    width=0.45\textwidth,
    xlabel={$x_1$},
    ylabel={$x_2$},
    xmin=-0.2, xmax=1.2,
    ymin=-0.2, ymax=1.2,
    label style={font=\scriptsize},
    ticklabel style={font=\tiny},
    legend pos={north west},
    legend style={font=\scriptsize},
    ]
\addplot [color=blue,solid,ultra thick] coordinates {
    (0,0)
    (0,0.5)
    (1,0.5)
    (1,1)};
\addlegendentry{$C_1$-minimizers}
\addplot [domain=0:1,solid,color=red,ultra thick]{x};
\addlegendentry{$C_2$-minimizers}
\addplot [only marks,draw=black] coordinates {
    (0,0)
    (0.5,0.5)
    (1,1)};
\end{axis}
\end{tikzpicture}
\caption{Minimizers for the vector optimization problem associated with ordering cone $C_i$ for player $i$ in Example~\ref{ex:shared-multiple}.  The intersection, and thus the Nash equilibria, are highlighted as solid black points.}
\label{fig:shared-multiple}
\end{figure}
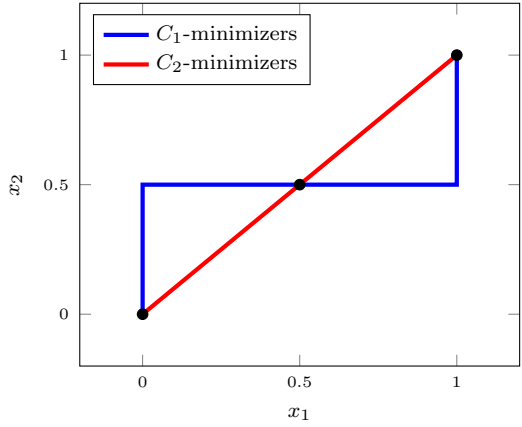
\end{example}

\section{Generalized Nash Games as Vector Optimization}\label{sec:generalized}

In this section, we consider generalized Nash games, i.e. where we do not assume shared constraints as in Section~\ref{sec:shared}. In Section~\ref{sec:genNash}, the set of generalized Nash equilibria is shown to coincide with the intersection of Pareto optimal points of $N$ vector optimization problems.
In Section~\ref{sec:genNash-shared}, a relationship between generalized and shared-constraint games constructed from the individual constraints is deduced. In particular, in Theorem~\ref{thm:nash-generalized} and Corollary~\ref{Cor:Nashii} a superset and a subset of the set of generalized Nash equilibria is derived based on Pareto optimal points of certain vector optimization problems corresponding to these shared-constraint games. 
We illustrate these results with two examples that we study in both sections.

\subsection{Characterization of generalized Nash equilibria}\label{sec:genNash}
Within this section, we study generalized Nash games, that is, $N$ player non-cooperative games without the requirement that the constraint sets for all players coincide.  That is, each player $i$ has a constraint set $\bbx_i \subseteq \prod_{j = 1}^N \xcal_j$ on the joint strategy of all players.  With this modification, a generalized Nash game is one in which, given all other players fix their strategies $x_{-i}^* \in \prod_{j \neq i} \xcal_j$, player $i$ seeks to optimize
\begin{equation}\label{eq:game-generalized}
\min_{x_i \in \xcal_i} \left\{f_i(x_i,x_{-i}^*) \; | \; (x_i,x_{-i}^*) \in \bbx_i \right\} \quad \forall i = 1,...,N.
\end{equation}
This construction leads to an updated definition of a Nash equilibrium.
\begin{definition}\label{defn:nash-generalized}
Consider the generalized game with $N$ players described by~\eqref{eq:game-generalized}.  
That is, player $i$ has strategies in the linear space $\xcal_i$, cost function $f_i: \prod_{j = 1}^N \xcal_j \to \bbr$, and such that the joint strategies must satisfy the feasibility condition $x \in \bbx_i \subseteq \prod_{j = 1}^N \xcal_j$.  
The joint strategy $x^* \in \bigcap_{i = 1}^N \bbx_i$ is called a \textbf{\emph{(generalized) Nash equilibrium}} if, for any player $i$, $f_i(x_i,x_{-i}^*) \geq f_i(x^*)$ for any strategy $x_i \in \xcal_i$ such that $(x_i,x_{-i}^*) \in \bbx_i$.  The set of all Nash equilibria will be denoted by $\NE(f,(\bbx_i)_{i = 1}^N)$.
\end{definition}

Note again, that the joint strategy $x^* \in \bigcap_{i = 1}^N \bbx_i$ is called a Nash equilibrium if, for any player $i$, 
\begin{equation}\label{eq:Nashii}
	x_i^* \in \argmin_{x_i \in \xcal_i} \{f_i(x_i,x_{-i}^*) \; | \; (x_i,x_{-i}^*) \in \bbx_i\}.
\end{equation}

Immediately, we are able to provide the second main result of this paper which characterizes the set of generalized Nash equilibria as the intersection of $N$ vector optimization problems.  The result of Corollary~\ref{cor:intersec_gen} can be readily compared to Corollary~\ref{cor:nash-shared} for shared constraint games.  
\begin{corollary}\label{cor:intersec_gen}
Consider the generalized non-cooperative game~\eqref{eq:game-generalized}.
The strategy $x^*$ is a Nash equilibrium if and only if $x^* \in C_i\mbox{-}\argmin\{(x,f(x)) \; | \; x \in \bbx_i\}$ for every $i = 1,...,N$, i.e.,
\begin{align*}
\NE(f,(\bbx_i)_{i = 1}^N) =  \bigcap_{i = 1}^N  C_i\mbox{-}\argmin\{(x,f(x)) \; | \; x \in \bbx_i\}.
\end{align*}
\end{corollary}
As the proof of Corollary~\ref{cor:intersec_gen} follows comparably to that of Theorem~\ref{thm:nash-shared}, we omit its proof.
Intuitively this result follows since the set of Nash equilibria are all of the fixed points of the best response functions and, as discussed in Remark~\ref{Rem:BestResp}, $C_i\mbox{-}\argmin\{(x,f(x)) \; | \; x \in \bbx_i\}$ coincides with the graph of player $i$'s best response function.
\endproof

The following two examples will illustrate Corollary~\ref{cor:intersec_gen}. 
Both examples are $2$ player games with linear objectives and constraints.  This allows us to utilize methods for linear vector optimization to construct the set of optimizers.  Details on this are given in Remark~\ref{Rem:linVOP} below.  
As far as the authors are aware, no prior methods exist which are able to compute the set of all Nash equilibria in these examples.

\begin{example}
\label{ex:generalized}
Consider the $N=2$ player generalized game with linear objectives and constraints
\begin{align}
\nonumber x_1^* &\in \argmax_{x_1 \in \bbr^2}\{2x_{11} + x_{12} \; | \; (x_1,x_2^*) \in \bbx_1\}, 
\qquad x_2^* \in \argmax_{x_2 \in \bbr^2}\{2x_{21} + 3x_{22} \; | \; (x_1^*,x_2) \in \bbx_2\},
\end{align}
\begin{align}
\label{ex:genX1} \bbx_1&=\left\{x\in \bbr^4 \; | \; \left(\begin{array}{cccc} 4 & 1 & -\frac{16}{3} & -\frac{1}{3} \\ 1 & 2 & 0 & 0 \end{array}\right)\left(\begin{array}{c} x_1 \\ x_2 \end{array}\right) \leq \left(\begin{array}{c} 0 \\ 5 \end{array}\right),~\begin{array}{l} x_1\in [0,5]\times[0,2.5], \\ x_2\in [0,1.5]\times[0,6] \end{array} \right\},\\
\label{ex:genX2} \bbx_2&=\left\{x\in \bbr^4 \; | \; \left(\begin{array}{cccc} 0 & 0 & 4 & 1 \\ 15 & -10 & 1 & 2 \end{array}\right)\left(\begin{array}{c} x_1 \\ x_2 \end{array}\right) \leq \left(\begin{array}{c} 6 \\ 0 \end{array}\right),~\begin{array}{l} x_1\in [0,5]\times[0,2.5], \\ x_2\in [0,1.5]\times[0,6] \end{array} \right\},
\end{align}
in which player $1$ plays strategy $x_1 = (x_{11},x_{12}) \in \xcal_1=\bbr^2$ and player $2$ plays strategy $x_2 = (x_{21},x_{22}) \in \xcal_2=\bbr^2$.
Thus we consider a generalized game~\eqref{eq:game-generalized} with $f_1(x)=-(2x_{11} + x_{12})$, $f_2(x)=-(2x_{21} + 3x_{22})$, and where the constraint sets are given by~\eqref{ex:genX1} and~\eqref{ex:genX2}.

In order to solve for the set of Nash equilibria, we solve (by Corollary~\ref{cor:intersec_gen}) the corresponding
two linear vector optimization problems 
$
C_i\mbox{-}\argmin \{(x,f(x)) \; | \; x \in \bbx_i\}
$
with polyhedral convex cones $C_i$ for $i=1,2$. 
The intersection of these two sets of Pareto optimal points, and thus the set of Nash equilibria, is given by 
\[
\NE(f,(\bbx_i)_{i = 1}^2) = \{x^1,...,x^4\}
\]
with equilibrium points
$x^1 = ((0,0) \, , \, (0,0))$, 
$x^2 = ((0,2) \, , \, (0,6))$, 
$x^3 = ((1,2) \, , \, (1,2))$, 
and $x^4 \approx ((1.1876,1.9062) \, , \, (1.2481,0))$.
\end{example}
The following example shows that the proposed method can also find a set of Nash equilibria consisting of infinitely many equilibria.
\begin{example}
\label{ex:generalized-2}
Consider now the following modification of the $N=2$ player game given in Example~\ref{ex:generalized}:
\begin{align*}
x_1^* &\in \argmax_{x_1 \in \bbr^2}\{2x_{11} + x_{12} \; | \; (x_1,x_2^*) \in \hat\bbx_1\},
\qquad x_2^* \in \argmax_{x_2 \in \bbr^2}\{2x_{21} + 3x_{22} \; | \; (x_1^*,x_2) \in \hat\bbx_2\}, 
\end{align*}
with $\hat\bbx_1 = \bbx_1 \cap \{x \in \bbr^4 \; | \; 4x_{21}+x_{22} \leq 6\}$ and $\hat\bbx_2 = \bbx_2 \cap \{x \in \bbr^4 \; | \; 4x_{11} + x_{12} \leq \frac{16}{3}x_{21} + \frac{1}{3}x_{22}\}$ where $\bbx_1$ and $\bbx_2$ are defined in~\eqref{ex:genX1} and~\eqref{ex:genX2}. As in Example~\ref{ex:generalized}, player $1$ plays strategy $x_1 \in \bbr^2$ and player $2$ plays strategy $x_2 \in \bbr^2$. The objectives $f_1,f_2$ and the cones $C,C_1,C_2$ are as in Example~\ref{ex:generalized}, only the constraint sets $\hat\bbx_i$ are different from Example~\ref{ex:generalized}.

As with Example~\ref{ex:generalized}, in order to solve for the set of Nash equilibria, we solve (by Corollary~\ref{cor:intersec_gen}) the corresponding two linear vector optimization problems
$C_i\mbox{-}\argmin\{(x,f(x)) \; | \; x \in \hat\bbx_i\}$
with polyhedral convex cones $C_i$ for $i=1,2$.  
The intersection of these two sets of Pareto optimal points, and thus the set of Nash equilibria, is given by
\[
\NE(f,( \hat\bbx_i)_{i = 1}^2) = \{x^1\} \cup P_2 \cup P_3 \cup \{x^4\}
\]
where
$P_2 = \co\{x^2 \, , \, x^5\}$ 
and $P_3 = \co\{x^3 \, , \, x^5\}$ 
for extremal points $x^1,...,x^4$ as in Example~\ref{ex:generalized} and $x^5 = ((0,2.5) \, , \, (0.125,5.5))$.
Here, we denote by $\co A$ the convex hull of a set $A$.
\end{example}

\begin{remark}\label{Rem:linVOP} 
Both Example~\ref{ex:generalized} and~\ref{ex:generalized-2} considered in this section are Nash games with linear objectives and linear constraints with $ \xcal_i=\bbr^2$. 
This implies that for each $i$, all of the considered vector optimization problems $C_i\mbox{-}\argmin \{(x,f(x)) \; | \; x \in\bbx_i\}$ are linear vector optimization problems.
However, the cones $C_i$ are unusual in the sense that they contain lines and have an empty interior. Because of this, standard solvers might not be applied directly. However, these linear vector optimization problems can be transformed into multi-objective linear programs (i.e.\ into linear vector optimization problems with a natural ordering cone) 
by multiplying the objective function with the matrix $Z_i$ containing the $m_i$ generating vectors of the positive dual cone of $C_i$. The particular structure of $C_i$ will always lead to $Z_i^\T(x,f(x))=(x_{-i}, -x_{-i}, f_i(x))$. That is, 
\[
\bbr^{m_i}_+\mbox{-}\argmin \{(x_{-i}, -x_{-i}, f_i(x)) \; | \; x \in\bbx_i \} = C_i\mbox{-}\argmin \{(x,f(x)) \; | \; x \in\bbx_i \},
\]
see \cite[Lemma~2.3.4]{SNT85} for pointed polyhedral convex cones though this can be extended to any polyhedral convex cone (by noting that $C_i = \{x \; | \; Z_i^\T x \geq 0\}$ and therefore $\hat y - y \in C_i$ if and only if $Z_i^\T(\hat y - y) \geq 0$).  
Note that both of the cones~$C_i$ in Example~\ref{ex:generalized} and~\ref{ex:generalized-2} have positive dual cones with $m_i=5$ generating vectors. Hence, in this case, this transformation turns a linear vector optimization problem with a $6$-dimensional image space into a standard multi-objective linear program with a $5$-dimensional image space (further dimension reduction is discussed in~\cite{HFR22}), which can be solved e.g.\ with \cite{Armand93,VanTu17,tohidi2018adjacency}. 
Note that for the results of this paper the whole set of minimizers, i.e.\ the set of all Pareto optimal points of the corresponding vector optimization problems/multi-objective linear programs, is needed.  This can be seen clearly in Example~\ref{ex:generalized-2} as there are entire line segments that are Nash equilibria. Algorithms such as \cite{steuer73,HLR14,RUV17} provide only the set of all efficient extreme points of the problems considered here (though this is only true for \cite{HLR14,RUV17} because of the particular structure of the objective function used here). 
In contrast, the algorithms of \cite{Armand93,VanTu17,tohidi2018adjacency} 
provide the set of all Pareto optimal points, i.e., also the efficient faces.
\end{remark}

\subsection{Relation to shared constraint games}\label{sec:genNash-shared}
We continue our study of generalized Nash games (see Definition~\ref{defn:nash-generalized}) by finding meaningful shared constraint games that produce a sandwich principle on the set of Nash equilibria.  
This is in contrast to the characterization of the set of Nash equilibria from Corollary~\ref{cor:intersec_gen} in which a set of $N$ vector optimization problems are considered.
Though it may seem counterintuitive to consider such a game when Corollary~\ref{cor:intersec_gen} already provides a procedure for determining the set of Nash equilibria for the generalized game, we refer the interested reader to~\cite{braouezec2021economic}.  For instance, consider~\cite[Section 3.1]{braouezec2021economic} in which nations participate in a game to determine their pollution levels (e.g., carbon emissions) subject to individual constraints on desired industrial production. Assuming linear objectives, this generalized game can be written as a linear game and thus solved directly using the methodology proposed in Remark~\ref{Rem:linVOP}. However, as proven in Proposition 2 of~\cite{braouezec2021economic}, this generalized game may fail to admit any Nash equilibria at all.  But enforcing a shared constraint for all nations (e.g., through treaty obligations) can guarantee the existence of a Nash equilibrium and, as proven below in Theorem~\ref{thm:nash-generalized} and, independently, in Proposition 1 of~\cite{braouezec2021economic}, will always admit any generalized equilibria as a solution as well. 

Motivated thusly, we wish to consider shared constraint games associated with the generalized Nash game $(f,(\bbx_i)_{i = 1}^N)$. As expected from the pollution game of~\cite[Section 3.1]{braouezec2021economic}, the lack of a shared feasibility constraint in~\eqref{eq:Nashii} makes it impossible to find an appropriate (meaningful) set $\bbx$ that would provide an equivalence between the solutions of the generalized Nash game~\eqref{eq:Nashii} and the \emph{single} vector optimization problem~\eqref{eq:vop}. 
Instead, for generalized games, we find that an appropriate choice of the constraint set $\bbx$ can lead to sufficient, another choice of $\bbx$ to necessary conditions for being a generalized Nash equilibrium; this is proven in Theorem~\ref{thm:nash-generalized} below. Thus, by solving an appropriate vector optimization problem, one can obtain a superset of the set of all generalized Nash equilibrium. Whereas by solving a different appropriate vector optimization problem, one can obtain a subset of the set of all generalized Nash equilibrium. This will be made precise in the following theorem, which is our third main result of this paper.

\begin{theorem}\label{thm:nash-generalized}
Consider the generalized non-cooperative game \eqref{eq:game-generalized}.  
\begin{enumerate}
\item Let $\overline\bbx \subseteq \bigcap_{i = 1}^N \bbx_i$.  The strategy $x^* \in \overline\bbx$ is a Nash equilibrium only if it is a minimizer of~\eqref{eq:vop} with constraint set $\overline\bbx$, i.e., 
\begin{align}\label{superset}
\NE(f,(\bbx_i)_{i = 1}^N) \cap \Big\{\overline\bbx\Big\} \subseteq C\mbox{-}\argmin \{(x,f(x)) \; | \; x \in \overline\bbx\}, 
\end{align}
with ordering cone $C$ defined as in~\eqref{eq:C}. 
In particular, if we set $\overline\bbx = \bigcap_{i = 1}^N \bbx_i$, then 
 every Nash equilibrium is a minimizer of~\eqref{eq:vop} with constraint set $\overline\bbx$, i.e.,
\begin{align}\label{superset2}
\NE(f,(\bbx_i)_{i = 1}^N) \subseteq C\mbox{-}\argmin \{(x,f(x)) \; | \; x \in\bigcap_{i = 1}^N \bbx_i\}.
\end{align}
\item Let $\underline\bbx \supseteq \bigcup_{i = 1}^N \bbx_i$.  The strategy $x^*$ is a Nash equilibrium if it is feasible for the generalized game and a minimizer of~\eqref{eq:vop} with constraint set $\underline\bbx$ and ordering cone $C$ as defined in~\eqref{eq:C}, i.e., 
\begin{align}\label{subset}
\NE(f,(\bbx_i)_{i = 1}^N) \supseteq \Big\{C\mbox{-}\argmin \{(x,f(x)) \; | \; x \in \underline\bbx\}\Big\}\cap \Big\{\bigcap_{i = 1}^N \bbx_i\Big\}.
\end{align}
\end{enumerate}
\end{theorem}
\proof{Proof}
Recall from Definition~\ref{defn:pareto} that $x^* \in C\mbox{-}\argmin \{(x,f(x)) \; | \; x \in \bbx\}$ if and only if for any $x \in \bbx$ such that $(x,f(x)) \leq_C (x^*,f(x^*))$ then $(x,f(x)) \geq_C (x^*,f(x^*))$.
\begin{enumerate}
\item Consider $x^* \in \overline\bbx$ to be a Nash equilibrium of the game.  By definition of the Nash equilibrium, for any player $i$, if $x \in \bbx_i$ with $x_{-i} = x_{-i}^*$ then $f_i(x) \geq f_i(x^*)$. In particular, this is true for $x \in \overline\bbx \subseteq \bbx_i$ as well.  Therefore if $(x,f(x)) \leq_C (x^*,f(x^*))$ then $x_{-i} = x_{-i}^*$ for some $i$ which from a Nash equilibrium we know $f_i(x) \geq f_i(x^*)$, i.e., $(x,f(x)) \geq_C (x^*,f(x^*))$.

Equation~\eqref{superset2} trivially follows as $\NE(f,(\bbx_i)_{i = 1}^N) \subseteq \bigcap_{i = 1}^N \bbx_i$.

\item Consider $x^*$ to be a minimizer of the vector optimization problem with constraints $\underline\bbx$ that is feasible for the original game (i.e., $x^* \in \bigcap_{i = 1}^N \bbx_i$).  By definition of the Pareto minimizers and Lemma~\ref{lemma:decomposition}, for any player $i$, if $x \in \underline\bbx$ with $x_{-i} = x_{-i}^*$ then $f_i(x) \geq f_i(x^*)$. In particular, this is true for $x \in \bbx_i \subseteq \underline\bbx$ as well.  This immediately implies $x^*$ is a Nash equilibrium of the game.
\end{enumerate}
\endproof

\begin{remark}\label{rem:nash-generalized}
Within the context of Theorem~\ref{thm:nash-generalized}, the most useful sets to consider are $\overline\bbx = \bigcap_{i = 1}^N \bbx_i$ and $\underline\bbx = \cl\co\bigcup_{i = 1}^N \bbx_i$ (i.e., the closed and convex hull of $\bigcup_{i = 1}^N \bbx_i$).
In particular, if we additionally assume that $\bbx_i$ is closed and convex for every $i$,
 then the choice of $\overline\bbx = \bigcap_{i = 1}^N \bbx_i$ in~\eqref{superset} leads  on the one hand to the largest possible set of Nash equilibria in~\eqref{superset} (as $\NE(f,(\bbx_i)_{i = 1}^N) \subseteq \bigcap_{i = 1}^N \bbx_i$ ensures $\NE(f,(\bbx_i)_{i = 1}^N) \cap \{\overline\bbx\} =\NE(f,(\bbx_i)_{i = 1}^N)$ for this choice of $\overline\bbx$), and on the other hand to the smallest possible superset in~\eqref{superset} amongst all closed and convex constraint sets.
Similarly, and again assuming that $\bbx_i$ is closed and convex for every $i$,
 the choice of $\underline\bbx = \cl\co\bigcup_{i = 1}^N \bbx_i$ in~\eqref{subset} leads to the largest possible subset in~\eqref{subset} amongst all closed and convex constraint sets.
 
We will occasionally refer to $C\mbox{-}\argmin \{(x,f(x)) \; | \; x \in\bigcap_{i = 1}^N \bbx_i\}$ appearing in~\eqref{superset2} as the ``intersection game'', and to $C\mbox{-}\argmin \{(x,f(x)) \; | \; x \in  \cl\co\bigcup_{i = 1}^N \bbx_i\}$ appearing in~\eqref{subset} when setting $\underline\bbx = \cl\co\bigcup_{i = 1}^N \bbx_i$ as the  ``union game'' of the generalized Nash game \eqref{eq:game-generalized}. The names are justified as the  ``intersection game'' and the   ``union game'' are both shared constraint Nash games by Theorem~\ref{thm:nash-shared}. The ``intersection game'' as a necessary condition for the generalized Nash equilibrium encoded in~\eqref{superset2} is independently presented in a purely game theoretic setting (i.e.\ without the connection to vector optimization) in Proposition 1 of~\cite{braouezec2021economic}.
\end{remark}

By Lemma~\ref{lemma:decomposition}, the following corollary is immediate.
\begin{corollary}\label{Cor:Nashii} 
It holds
\begin{align}\label{superseti}
	\NE(f,(\bbx_i)_{i = 1}^N)\subseteq \bigcap_{i = 1}^N C_i\mbox{-}\argmin \{(x,f(x)) \; | \; x \in \bigcap_{i = 1}^N \bbx_i\}.
\end{align}
Furthermore, for any choice of $\underline\bbx \supseteq \bigcup_{i = 1}^N \bbx_i$, it holds
\[
	\NE(f,(\bbx_i)_{i = 1}^N) \supseteq \Big\{\bigcap_{i = 1}^N C_i\mbox{-}\argmin \{(x,f(x)) \; | \; x \in\underline\bbx\} \Big\}\cap \Big\{\bigcap_{i = 1}^N \bbx_i\Big\}.
\]
\end{corollary}

Let us now return to the superset relation~\eqref{superset2} in Theorem~\ref{thm:nash-generalized}, respectively relation~\eqref{superseti} in Corollary~\ref{Cor:Nashii}.  We now want to characterize those Pareto points in the superset that are not generalized Nash equilibria. 
This will  provide an economic interpretation for the discrepancy between the set of generalized Nash equilibria and the set of Pareto optimal points.

Note that, in the generalized Nash game, player $i$ seeks to minimize her cost function $f_i$, given the strategy $x_{-i}^* \in \prod_{j \neq i} \xcal_j$ of the other players in \eqref{eq:game-generalized}. 
Player $i$ takes, however, only her own constraint set $\bbx_i$ into account. So her set of feasible strategies can produce joint strategies $(x_i,x_{-i}^*) \notin \bbx_j$ for some $j\neq i$.  Thus, it can very well happen that 
\begin{align}\label{ego}
\min_{x_i \in \xcal_i} \{f_i(x_i,x_{-i}^*) \; | \; (x_i,x_{-i}^*) \in \bbx_i\}<\min_{x_i \in \xcal_i} \{f_i(x_i,x_{-i}^*) \; | \; (x_i,x_{-i}^*) \in \bigcap_{i = 1}^N \bbx_i\}.
\end{align}
So the setup of a \emph{generalized} Nash game takes on an egocentric point of view. Player~$i$ does not care about the consequences of changing her strategy with respect to feasibility of the others. She cares only about her own feasibility and minimizing her own cost. \emph{So the setup of a generalized Nash game excludes possible equilibria, where a player could decrease her cost only by strategies that are not feasible for other players.} 
These excluded points are indeed equilibria, but of a  shared constraint Nash game that coincides with the generalized Nash game, but considers in her best responds problem only the set of strategies that lead to joint strategies that stay feasible for all. Recall from Theorem~\ref{thm:nash-shared} that the Nash equilibria $\NE(f,\bigcap_{i = 1}^N \bbx_i)$ of this shared constraint game coincide with the set of Pareto optimal points of the vector optimization problem~\eqref{eq:vop} with constraint set $\bbx = \bigcap_{i = 1}^N \bbx_i$ considered in~\eqref{superset2}, i.e.
\[
	\NE(f,\bigcap_{i = 1}^N \bbx_i)=C\mbox{-}\argmin \{(x,f(x)) \; | \; x \in   \bigcap_{i = 1}^N \bbx_i\}.
\]
This explains why the set of Pareto optimal points in~\eqref{superset2} can be larger than the set of generalized Nash equilibria and it provides an economic interpretation for when this happens.

We wish to conclude this section by illustrating the results relating generalized Nash games to shared constraint games. In Example~\ref{ex:generalized-cont} below, we will revisit Example~\ref{ex:generalized} to show that the superset relation~\eqref{superset2} in Theorem~\ref{thm:nash-generalized}, respectively relation~\eqref{superseti} in Corollary~\ref{Cor:Nashii}, can be strict. 

Example~\ref{ex:generalized-cont} and~\ref{ex:generalized-2-cont} also show that the subset relation~\eqref{subset} in Theorem~\ref{thm:nash-generalized} (respectively in Corollary~\ref{Cor:Nashii}) can be strict. It can even lead to an empty set on the right hand side of~\eqref{subset} as demonstrated in Example~\ref{ex:generalized-cont}.
Example~\ref{ex:generalized-2-cont} revisits Example~\ref{ex:generalized-2} and is a modification of Example~\ref{ex:generalized-cont}, where the subset relation~\eqref{subset} is strict, but does not lead to an empty set.

\begin{example}
\label{ex:generalized-cont}
Recall the 
generalized game from Example~\ref{ex:generalized}. 
Let us now consider the ``intersection game'', i.e.\ the vector optimization problem 
$C\mbox{-}\argmin \{(x,f(x)) \; | \; x \in\bigcap_{i = 1}^2 \bbx_i\}$
appearing in~\eqref{superset2} in Theorem~\ref{thm:nash-generalized}, where the cone $C\subseteq\bbr^6$ is given as in~\eqref{exC_2}. 
The Pareto optimizers of this ``intersection game'' provide a superset of the Nash equilibria of the generalized game by Theorem~\ref{thm:nash-generalized}.
In order to solve the vector optimization problem with non-convex ordering cone $C$, we solve (by Lemma~\ref{lemma:decomposition}) the corresponding
two linear vector optimization problems with convex ordering cones $C_i$ and take an intersection of their Pareto optimal points leading to
$
	P^\cap=\{x_1\} \cup \bigcup_{k = 2}^5 P_k,
$ where $P_2,P_3$ are as in Example~\ref{ex:generalized-2} and 
$P_4 = \co\{x^3 \, , \, x^6\}$, 
$P_5 = \co\{x^4 \, , \, x^6\}$ 
for extremal points $x^1,...,x^4$ as in Example~\ref{ex:generalized}, $x^5$ as in Example~\ref{ex:generalized-2}, and
$x^6 = ((1.175,1.9125) \, , \, (1.5,0))$.
Condition~\eqref{ego} characterizes exactly those points in the superset $P^\cap$ that are not generalized Nash equilibria. Removing those points from $P^\cap$ leads to the set of generalized Nash equilibria of this game as already computed in Example~\ref{ex:generalized}.

For completeness, let us now also consider the ``union game'', i.e.\ the vector optimization problem $C\mbox{-}\argmin \{(x,f(x)) \; | \; x \in \underline\bbx\}$ appearing in the subset relation~\eqref{subset} in Theorem~\ref{thm:nash-generalized}  for $\underline\bbx = \cl\co\bigcup_{i = 1}^2 \bbx_i$. While the ``union game'' has Pareto optimal points consisting of 2 efficient faces (a plane and a line) including $4$ extremal points, the intersection with the feasibility condition $\bigcap_{i = 1}^2 \bbx_i$ in~\eqref{subset}, leads to an empty set on the right hand side of~\eqref{subset} in this example. We wish to highlight that, though the ``union game'' is a convex, shared constraint game (and thus existence of a Nash equilibrium is guaranteed as found computationally here; see also Theorem 1 of~\cite{rosen65}), the feasibility condition can lead to an empty set of equilibria in the sufficient condition~\eqref{subset}.
\end{example}

\begin{example}
\label{ex:generalized-2-cont}
Consider again the modified $N=2$ player game given in Example~\ref{ex:generalized-2}.
Consider first its related ``union game'', i.e., problem $C\mbox{-}\argmin \{(x,f(x)) \; | \; x \in \underline\bbx\}$ appearing in the subset relation~\eqref{subset} in Theorem~\ref{thm:nash-generalized} for $\underline\bbx = \cl\co\bigcup_{i = 1}^2 \hat\bbx_i$.
This ``union game'' has a single line segment minimizer that is also feasible for the generalized game (i.e., an element of $\bigcap_{i = 1}^2 \hat\bbx_i$) and is given by
$P^\cup = P_2$, 
where $P_2$ is as in Example~\ref{ex:generalized-2}, i.e., $P_2 = \co\{x^2 \, , \, x^5\}$ for
$x^2 = ((0,2) \, , \, (0,6))$ and
$x^5 = ((0,2.5) \, , \, (0.125,5.5))$. 
Thus, in contrast to Example~\ref{ex:generalized-cont}, one obtains already a set of Nash equilibria from the ``union game'' by applying relation~\eqref{subset} in Theorem~\ref{thm:nash-generalized}. Thus, 
$
	P_2\subseteq \NE(f,(\hat\bbx_i)_{i = 1}^N).
$

Note that, as by construction, $\hat\bbx_1 \cap \hat\bbx_2 = \bbx_1 \cap \bbx_2$, the corresponding ``intersection game'' is identical to that of Example~\ref{ex:generalized-cont}. Therefore it has the same set of Pareto optimal points $P^\cap$ 
providing a superset of the generalized Nash equilibria of this game by~\eqref{superseti} in Corollary~\ref{Cor:Nashii}. 
\end{example}

\section{Vector-Valued Nash Games as Vector Optimization}\label{sec:vector}

In this section, we consider (generalized) non-cooperative games in which the $N$ players have vector-valued payoffs.  This extends the settings of Section~\ref{sec:shared} and~\ref{sec:generalized} by allowing for the space of costs to be \emph{partially ordered}.  As formalized below, and as stated in e.g.~\cite{shapley1959,bade2005}, the Nash equilibria in such a setting are w.r.t.\ Pareto efficient costs. The problem of finding any of these Nash equilibria is typically extremely challenging and often requires the consideration of an infinite number of scalarized problems~(see \cite{Corley85,bade2005} as well as Remark~\ref{rem:scalarization} below).  Such problems naturally arise due to incomplete preferences in the costs, e.g., costs in multiple assets with transaction costs between these assets.  

In Section~\ref{sec:vector-Nash}, we will present the primary theoretical results of this section extending, e.g., Corollary~\ref{cor:intersec_gen} to provide Corollary~\ref{cor:nash-vector} which proves that the same characterization of Nash games as Pareto optimal points of a certain vector optimization problem also holds for vector-valued games.  Within this context we also provide a brief introduction to vector-valued Nash games.  In Section~\ref{sec:vector-ex}, we will illustrate these results with a numerical example, which also demonstrates the computational advantages of the Pareto reformulation of vector-valued Nash games.

\subsection{Characterization of Nash equilibria for vector-valued games}\label{sec:vector-Nash}

Within this section, we study vector-valued Nash games, that is, $N$ player non-cooperative games in which the costs are partially ordered.  We will present these results in the context of generalized Nash games as in Section~\ref{sec:generalized}. 

Specifically, each player $i$ has a cost function $f_i: \prod_{j = 1}^N \xcal_j \to \ycal_i$ mapping to some partially ordered space $\ycal_i$ with associated convex ordering cone $K_i \subseteq \ycal_i$.
As in Section~\ref{sec:generalized}, each player $i$ also has a constraint set $\bbx_i \subseteq \prod_{j = 1}^N \xcal_j$ on the joint strategy of all players.  With this setup, the Nash game is one in which, given all other players fix their strategies $x_{-i}^* \in \prod_{j \neq i} \xcal_j$, player $i$ seeks to optimize
\begin{equation}\label{eq:game-vector}
K_i\mbox{-}\vmin_{x_i \in \xcal_i} \left\{f_i(x_i,x_{-i}^*) \; | \; (x_i,x_{-i}^*) \in \bbx_i\right\} \quad \forall i = 1,...,N.
\end{equation}
Within~\eqref{eq:game-vector}, the minimum is in the sense of Pareto optimality (see Definition~\ref{defn:pareto}).  This construction leads to the notion of the Nash equilibrium for vector-valued games as is provided in, e.g.,~\cite{shapley1959,bade2005}. The Nash equilibrium in vector-valued games is also called, e.g., the Shapley equilibrium in \cite{HL18}, the Pareto equilibrium in \cite[Chapter 11]{voorneveld1999potential}, and the Pareto Nash equilibrium in~\cite{JM07}.
\begin{definition}\label{defn:nash-vector}
Consider the generalized vector-valued Nash game with $N$ players described by~\eqref{eq:game-vector}.  
That is, player $i$ has strategies in the space $\xcal_i$, cost function $f_i: \prod_{j = 1}^N \xcal_j \to \ycal_i$ with ordering cone $K_i \subseteq \ycal_i$, and such that the joint strategies must satisfy the feasibility condition $x \in \bbx_i \subseteq \prod_{i = 1}^N \xcal_i$.  
The joint strategy $x^* \in \bigcap_{i = 1}^N \bbx_i$ is called a generalized \textbf{\emph{Nash equilibrium}} for a vector-valued game if, for any player $i$, 
$f_i(x_i,x_{-i}^*) \geq_{K_i} f_i(x^*)$ for any $x_i \in \xcal_i$ such that $(x_i,x_{-i}^*) \in \bbx_i$ and $f_i(x_i,x_{-i}^*) \leq_{K_i} f_i(x^*)$.
The set of all Nash equilibria will be denoted by $\NE(f,(\bbx_i)_{i = 1}^N)$; we simplify the notation to $\NE(f,\bbx)$ if this is a shared constraint game (i.e., if $\bbx = \bbx_i$ for every player $i$ as in Section~\ref{sec:shared}).
\end{definition}
Note again that the joint strategy $x^* \in \bigcap_{i = 1}^N \bbx_i$ is called a Nash equilibrium if, for any player $i$,
\[x_i^* \in K_i\mbox{-}\argmin_{x_i \in \xcal_i} \{f_i(x_i,x_{-i}^*) \; | \; (x_i,x_{-i}^*) \in \bbx_i\}.\]

\begin{remark}\label{rem:scalarization}
Let us shortly review the scalarization techniques used in the literature so far to solve vector-valued Nash games. For that purpose let us for now assume the following. Let $\ycal_i$ be a topological space with dual $\ycal_i^*$ and bilinear form $\langle \cdot,\cdot \rangle: \ycal_i^* \times \ycal_i \to \bbr$.  Let $K_i \subseteq \ycal_i$ be a closed and convex ordering cone.  Further, let $f_i(\cdot,x_{-i}^*)$ be $K_i$-convex for every $x_{-i}^* \in \prod_{j \neq i} \xcal_j$ , i.e., $f_i(\lambda x_i + (1-\lambda)\bar x_i,x_{-i}^*) \leq_{K_i} \lambda f_i(x_i,x_{-i}^*) + (1-\lambda) f_i(\bar x_i,x_{-i}^*)$ for any $\lambda \in [0,1]$ and $x_i,\bar x_i \in \xcal_i$, and let $\bbx_i$ be a convex set.  Then, as discussed in \cite[Chapter 11.2.1]{Jahn11}, $x_i^*$ solves the player $i$ best response function~\eqref{eq:game-vector} if and only if there exists some $w_i \in K_i^+ := \{k_i^* \in \ycal_i^* \; | \; \langle k_i^* , k_i \rangle \geq 0 \; \forall k_i \in K_i\}$ such that
\begin{equation}\label{eq:scalarization}
x_i^* \in \argmin_{x_i \in \xcal_i} \{\langle w_i , f_i(x_i,x_{-i}^*) \rangle \; | \; (x_i,x_{-i}^*) \in \bbx_i\}.
\end{equation}
Equation~\eqref{eq:scalarization} is often called the linear or weighted-sum scalarization of the vector optimization problem~\eqref{eq:game-vector}.
Because of the infinite number of linear scalarizations for a vector optimization problem, the traditional methods for finding Nash equilibria of vector-valued games generally fail. \cite[Section 4]{Corley85} explicitly discusses ``the impossible task of solving all possible scalarizations.''
\end{remark}

\begin{remark}\label{rem:scalar-game}
Note that we recover the typical (scalar-valued) definition for a Nash equilibrium (see Definitions~\ref{defn:nash-shared} and~\ref{defn:nash-generalized}) if $\ycal_i = \bbr$ and $K_i = \bbr_+$ for every player $i$.  This is due to the fact that $\bbr_+\mbox{-}\argmin_{x_i \in \xcal_i} \{f_i(x_i,x_{-i}^*) \; | \; (x_i,x_{-i}^*) \in \bbx_i\}$ provides the usual definition for the minimizers.
\end{remark}

In contrast to the scalarization techniques described in Remark~\ref{rem:scalarization}, we want to relate, similar to the preceding sections for real-valued games, vector-valued Nash games to Pareto optimal points of certain vector optimization problems. We will see that the additional assumptions stated in Remark~\ref{rem:scalarization} are not necessary for that. Example~\ref{ex:vector} below will show that the approach proposed here has clear computational advantages for game with linear objectives and constraints.

Consider again the vector optimization problem~\eqref{eq:vop} but with a modification to the ordering cone $C$ that accounts for the more general spaces of costs $\ycal_i$ with ordering cones $K_i$. Thus, 
for a constraint set $\bbx$, consider the following vector optimization problem and ordering cone
\begin{align}
\label{eq:vop-vector} C&\mbox{-}\vmin \{(x,f(x)) \; | \; x \in \bbx\} \\
\label{eq:C-vector} C &:= \{(x,y) \in \prod_{i = 1}^N \xcal_i \times \prod_{i = 1}^N \ycal_i \; | \; \exists i \in \{1,...,N\}: \; x_{-i} = 0_{\prod_{j \neq i} \xcal_j}, \; y_i \in K_i\}.
\end{align}
Similar to the cone defined in~\eqref{eq:C}, the ordering cone $C$ is \emph{non}-convex but can be decomposed into convex cones $C_i$ such that $C = \bigcup_{i = 1}^N C_i$ by defining
\begin{equation}
\label{eq:Ci-vector}
C_i := \left(\prod_{j = 1}^{i-1} \{0_{\xcal_j}\}\right) \times \xcal_i \times \left(\prod_{j = i+1}^N \{0_{\xcal_j}\}\right) \times \left(\prod_{j = 1}^{i-1} \ycal_j\right) \times K_i \times \left(\prod_{j = i+1}^N \ycal_j\right).
\end{equation}

As in Section~\ref{sec:generalized}, if $\bbx_i \neq \bbx_j$ for two players $i \neq j$ then, for vector-valued games, we can consider either a collection of $N$ vector optimization problems to characterize the set of Nash equilibria or consider two specific vector optimization problems to produce a sandwich principle.
In fact, in the following corollary -- the final main result of this work -- we will prove that the relations between the Nash equilibria and Pareto optimizers presented for real-valued games in Corollary~\ref{cor:intersec_gen} 
hold without modification for vector-valued games.
\begin{corollary}\label{cor:nash-vector}
Consider the generalized vector-valued non-cooperative game~\eqref{eq:game-vector}.  
The strategy $x^*$ is a Nash equilibrium if and only if $x^* \in C_i\mbox{-}\argmin\{(x,f(x)) \; | \; x \in \bbx_i\}$ for every $i = 1,...,N$, i.e.,
\[\NE(f,(\bbx_i)_{i = 1}^N) = \bigcap_{i = 1}^N C_i\mbox{-}\argmin\{(x,f(x)) \; | \; x \in \bbx_i\}.\] 
\end{corollary}

\begin{remark}\label{rem:vector-Nash-shared}
\begin{enumerate}
\item The proof is in total analogy to the proof of Corollary~\ref{cor:intersec_gen} 
by just replacing the cones with the ones defined in~\eqref{eq:Ci-vector}.
\item Under the shared constraint setting (i.e., $\bbx=\bbx_i = \bbx_j$ for every pair of two players $i,j$), 
the set of all Nash equilibria in Corollary~\ref{cor:nash-vector} 
can be reformulated as the single vector optimization problem~\eqref{eq:vop-vector} through an application of Lemma~\ref{lemma:decomposition}. Thus, the obtained result gives an equivalent characterization of the set of all shared constraint Nash equilibria of a vector-valued game as the Pareto optimal points of the vector optimization problem~\eqref{eq:vop-vector} with cone $C$ as given in~\eqref{eq:C-vector}, which is in total analogy to the real-valued case of Theorem~\ref{thm:nash-shared}.  
\item Theorem~\ref{thm:nash-generalized} and Corollary~\ref{Cor:Nashii}  hold in exactly the same formulation also for vector-valued games, only replacing the cones with the ones defined in~\eqref{eq:C-vector} and~\eqref{eq:Ci-vector}.
\end{enumerate}
\end{remark}

\subsection{Examples}\label{sec:vector-ex}
Within this section we will consider a shared constraint vector-valued Nash game to illustrate the results derived in Section~\ref{sec:vector-Nash}, in particular Corollary~\ref{cor:nash-vector} in the shared constraint case.  As with the examples of Section~\ref{sec:generalized}, this example is a 2 player game with linear objectives and constraints.  As such, we are able to utilize the methods for linear vector optimization to construct the set of optimizers as detailed in Remark~\ref{Rem:linVOP}.  This methodology is particularly noteworthy because, as far as the authors are aware, all prior works would require the study of an infinite number of scalarizations, see, e.g.,~\cite{bade2005}.  As noted in~\cite{Corley85}, such an approach is computationally intractable. We refer the interested reader to~\cite{JM07} for a survey of various refinements of the Nash equilibria for vector-valued games which  take ``into account the methodology of the scalarization which adds to the original problem new endogenous parameters.''
As such, as far as the authors are aware, there does not exist any other methodology which computes the set of all Nash equilibria for this vector-valued game.

\begin{example}
\label{ex:vector}
Consider the $N=2$ player vector-valued game with shared constraints 
\begin{align*}
x_1^* &\in \bbr^2_+\mbox{-}\argmax_{x_1 \in \bbr^2}\{(2x_{11},x_{12})^\T \; | \; (x_1,x_2^*) \in \bbx\},
\qquad x_2^* \in \bbr^2_+\mbox{-}\argmax_{x_2 \in \bbr^2}\{(2x_{21},3x_{22})^\T \; | \; (x_1^*,x_2) \in \bbx\},
\end{align*}
in which player 1 plays strategy $x_1 = (x_{11},x_{12}) \in \xcal_1 = \bbr^2$ and player 2 plays strategy $x_2 = (x_{21},x_{22}) \in \xcal_2 = \bbr^2$.  We set the shared constraints $\bbx$ to be the intersection of the constraints from Example~\ref{ex:generalized}, i.e.\ $\bbx = \bbx_1 \cap \bbx_2$ for $\bbx_1$ and $\bbx_2$ as in~\eqref{ex:genX1} and~\eqref{ex:genX2}.
Thus we consider a vector-valued game~\eqref{eq:vop-vector} with linear objectives $f_1(x) = -(2x_{11} \; , \; x_{12})^\T$, $f_2(x) = -(2x_{21} \; , \; 3x_{22})^\T$ both ordered by the natural ordering cone, i.e. the positive orthant $K_i = \bbr^2_+$ for $i=1,2$.

Note that the scalarization of this game with fixed weights $w_1 = (1,1)^\T$ and $w_2 = (1,1)^\T$ coincides with the ``intersection game'' of Examples~\ref{ex:generalized-cont} and~\ref{ex:generalized-2-cont} (see Remark~\ref{rem:scalarization}).  Therefore, $\{x^1\}$ and the sets $P_2,...,P_5$ (described in Examples~\ref{ex:generalized},~\ref{ex:generalized-2}, and~\ref{ex:generalized-cont}) must consist of Nash equilibria of this vector-valued game as well. However, as will be found below, these are not the only Nash equilibria for this game as some equilibria correspond to different weighted-sum scalarizations.
By applying Corollary~\ref{cor:nash-vector} under the shared constraint setting (as discussed in Remark~\ref{rem:vector-Nash-shared}), we will solve~\eqref{eq:vop-vector} with respect to the non-convex ordering cone 
\[C=(\bbr^2\times \{(0,0)\}\times \bbr^2_+\times \bbr^2)\cup(\{(0,0)\}\times \bbr^2\times \bbr^2\times \bbr^2_+)=C_1\cup C_2 \subseteq \bbr^8.\]
We will follow the same algorithmic approach as considered in Remark~\ref{Rem:linVOP} with updated generators $Z_1,Z_2$ for the positive dual cones of $C_1,C_2$ respectively; in this example there are $m_1 = m_2 = 6$ generating vectors for both of these positive dual cones. Thus, instead of considering two vector optimization problems with $8$-dimensional image spaces, one can equivalently solve two multi-objective linear programs with $6$-dimensional image spaces (with further dimension reduction possible as in~\cite{HFR22}) using e.g.\ \cite{VanTu17}.

The set of Nash equilibria using Corollary~\ref{cor:nash-vector}, is given by the union 
\[
	\NE(f,\bbx) := \NE_1 \cup \NE_2 \cup \NE_3 \cup \NE_4
\]
of the following $4$ convex polyhedrons $\NE_1 = \co\{x^1 \, , \, x^3 \, , \, x^4 \, , \, x^8\}$, $\NE_2 = \co\{x^2 \, , \, x^3 \, , \, x^5 \, , \, x^8\}$, 
$\NE_3 = \co\{x^3 \, , \, x^5 \, , \, x^6 \, , \, x^7\}$, and 
$\NE_4 = \co\{x^3 \, , \, x^4 \, , \, x^6\}$
for extremal points $x^1,...,x^6$ as given in the scalar games of Examples~\ref{ex:generalized},~\ref{ex:generalized-2}, and~\ref{ex:generalized-cont} and $x^7 = ((0,2.5) \, , \, (1.5,0))$, $x^8 = \Big(\Big(\frac{8}{55},\frac{78}{55}\Big) \, , \, (0,6)\Big)$.
As $x^1,...,x^6$ coincide exactly with extremal points found in scalar games of Examples~\ref{ex:generalized},~\ref{ex:generalized-2}, and~\ref{ex:generalized-cont}, all those solutions are equilibria w.r.t.\ the scalarizations $w_1 = w_2 = (1,1)^\T$.  The other two extremal solutions are new to this vector-valued game: $x^7$ is an equilibrium w.r.t.\ the scalarization $w_1 = (0,1)^\T$ and $w_2 = (1,0)^\T$; $x^8$ is an equilibrium w.r.t.\ the scalarization $w_1 = (2,1)^\T$ and $w_2 = (1,1)^\T$.
Note that even though, e.g., $\NE_4$ has extremal points all with scalarizations $w_1 = w_2 = (1,1)^\T$, not every point within that polyhedron $\NE_4$ is a Nash equilibrium w.r.t.\ those scalarizations.  This can be seen through a comparison of the above representation for the set of equilibria to that given for the ``intersection game'' in Examples~\ref{ex:generalized-cont} and~\ref{ex:generalized-2-cont}; that is, $\NE_4 \backslash P^\cap \neq \emptyset$.  This demonstrates a clear computational advantage to considering the vector optimization formulation presented herein over the more traditional scalarization approach because the convex combination of solutions for a fixed scalarization do not necessarily result in another Nash equilibrium with those weights.
\end{example}

\section*{Acknowledgment}
We are extremely grateful to Ta Van Tu, who computed the efficient faces in Examples~\ref{ex:generalized-cont},~\ref{ex:generalized-2-cont} and~\ref{ex:vector} with his implementation based on \cite{VanTu17}. 
We would also like to thank Andreas L\"ohne for a helpful discussion about non-convex ordering cones in vector optimization.
We would like to thank Benjamin Wei{\ss}ing for providing insights into the computation of the set of all efficient extreme points of a linear vector optimization problem with a convex ordering cone containing lines and having an empty interior in Bensolve (\cite{bensolve}) via polyhedral projections, which provides an alternative to the approach used and described  here in Remark~\ref{Rem:linVOP} of turning a linear vector optimization problem into a multi-objective linear program with a natural ordering cone.
We would like to thank Niklas Hey for the numerical implementation of Examples~\ref{ex:generalized} and~\ref{ex:generalized-2} based on a small correction of the algorithm proposed in \cite{tohidi2018adjacency}.

\bibliographystyle{plain}
\bibliography{biblio-Nash}

\end{document}